\documentclass[12pt]{article}

\usepackage[margin=1in]{geometry}
\usepackage{diagbox}
\usepackage{mathtools}
\usepackage{bbm}
\usepackage{latexsym}
\usepackage{epsfig}
\usepackage{amsmath,amsthm,amssymb,enumerate}
\usepackage{nicefrac}
\usepackage[usenames,dvipsnames]{xcolor}
\usepackage{graphicx}
\usepackage{color}

\def\red{\color{red}}

\definecolor{grey}{gray}{0.9}

\def\deg{\text{deg}}
\newcommand{\gap}[1]{\mbox{\hspace{#1 in}}}

\def\nn{\nonumber}
\def\a{\alpha} \def\b{\beta} \def\d{\delta} 
\def\e{\varepsilon}  \def\F{{\Phi}}  \def\g{\gamma}
\def\G{\Gamma}  \def\k{\kappa}
 \def\th{\theta}  \def\Th{\Theta}  
 \def\m{\mu}  \def\p{\pi}
   
\def\t{\tau} \def\om{\omega}  \def\Om{\Omega}

\newtheorem{theorem}{Theorem}
\newtheorem{lemma}[theorem]{Lemma}

\newtheorem{proposition}[theorem]{Proposition}

\newcommand{\ooi}{(1+o(1))}

\newcommand{\ol}[1]{\overline{#1}}
\newcommand{\wh}[1]{\widehat{#1}}

\newcommand{\rdup}[1]{{\left\lceil #1\right\rceil }}
\newcommand{\rdown}[1]{{\left\lfloor #1\right \rfloor}}

\newcommand{\brac}[1]{\left(#1\right)}

\newcommand{\sfrac}[2]{\frac{\scriptstyle #1}{\scriptstyle #2}}
\newcommand{\bfrac}[2]{\left(\frac{#1}{#2}\right)}
\def\half{\sfrac{1}{2}}
\def\cE{{\cal E}}
\newcommand{\rai}{\rightarrow \infty}
\newcommand{\ra}{\rightarrow}

\newcommand{\set}[1]{\left\{#1\right\}}
\def\sm{\setminus}
\def\seq{\subseteq}
\def\es{\emptyset}

\def\E{\mathbb{E}}
\def\Var{\mathbb{V}}
\def\Pr{\mathbb{P}}

\newcommand{\ignore}[1]{}

\def\cE{{\mathcal E}}

\def\cW{{\mathcal W}}

\newcommand{\beq}[2]{\begin{equation}\label{#1}#2\end{equation}}

\def\nn{\nonumber}

\def\tc{t_{\rm cov}}
\def\deg{{\rm deg}}
\def\dist{{\rm dist}}
\def\mcA{\mathcal{A}}

\begin{document}
\author{Colin Cooper\thanks{Department of Informatics, King's College, London WC2B 4BG, England. Research supported  at the University of Hamburg, by a  Mercator fellowship from DFG  Project 491453517}
\and
Alan Frieze\thanks{Department of Mathematical Sciences, Carnegie Mellon University, Pittsburgh PA 15213, USA. Research supported in part by NSF grant DMS1952285}
\and
Wesley Pegden\thanks{Department of Mathematical Sciences, Carnegie Mellon University, Pittsburgh PA 15213, USA. Research supported in part by NSF grant DMS1700365}
}
\date{\today}
\title{Cover time of random subgraphs of the hypercube}
\maketitle

\begin{abstract}
$Q_{n,p}$, the random subgraph of the $n$-vertex hypercube $Q_n$, is obtained by independently retaining each edge of $Q_n$ with probability $p$.
We give precise values for the cover time of $Q_{n,p}$ above the connectivity threshold.
\end{abstract}

\section{Introduction}

Let $Q_n$ be the hypercube with $n=2^d$ vertices  and $m=dn/2$ edges where $d=\log_2 n$ is the degree of any vertex.  Let $Q_{n,p}$ denote the random subgraph of the hypercube $Q_n$ with $n$ vertices where we retain each edge independently with probability $p$.
The threshold probability $p_c$ for connectivity in $Q_{n,p}$ has been the object of extensive study.
The original question as to whether connectivity enjoys a threshold property
was answered by Burtin in \cite{Burtin},
who proved that $p = 1/2$ is the threshold for connectedness.
This study  culminated in a proof by Bollob\'as \cite{BoHC}, in the random hypercube process, that  w.h.p. the hitting time for connectivity equals the hitting time  for minimum degree one. For more on this topic see e.g., \cite{FK}.

The cover time of a connected graph 
is the maximum over the start vertex  of the expected time  for a simple random walk to visit every vertex of the graph. There is a large literature on this subject  see for example \cite{AFill}, \cite{LPW}, including results \cite{CF1}--\cite{Hyp} on various models of random graphs  by the authors of this note.

Because  of its relationship with the Ehrenfest model of diffusion, the random walk on the hypercube has long been an object of study. Diaconis and Shahshahani \cite{DS} proved the existence of a cutoff phenomenon for the lazy walk
at $T= \sfrac{1}{4} d \log d$, and Diaconis, Graham and Morrison \cite{DGM} established the  rate of convergence
(of the total variation distance) to uniformity in the cutoff window.
Matthews \cite{Matt2} proved that the cover time of the hypercube $Q_n$ is $\tc =\ooi n \log n$.
The proof uses
results on the Matthews bound from the paper \cite{Matt1} by the same author.
This note  gives the w.h.p. cover time of $Q_{n,p}$, the random subgraph of the hypercube, above the connectivity threshold.

Denote $p= \half(1+\e)$, where $\e$ is a parameter  used subsequently with this unique meaning.
The condition for $Q_{n,p}$ to have minimum degree one, occurs w.h.p. when $d \e=\om$ where $\om \rai$ slowly.
As this is rather imprecise, and as our proofs are parameterized in terms of $d\e$, we will consider values of $p$ where $d \e \ge \th \log d$, for  some small  positive constant $\th$.
We assume henceforth that this holds, and thus $Q_{n,p}$ is connected w.h.p.

\begin{theorem}\label{TH1}
Let $p_c=\half(1+\th \log d/d)$ for  some small  positive constant $\th$.
Let $\tc(Q_{n,p})$ denote the cover time of $Q_{n,p}$.
For $p \ge p_c$, w.h.p.
\begin{equation}\label{tcov}
\tc(Q_{n,p})=  
\ooi\brac{\frac{p}{\log 2}\log \frac{2p}{2p-1}} n \log n.
\end{equation}
\end{theorem}
{\bf Remarks.}
If $p=(1/2)(1+\e)$ where $\e \ra 0$ then
\begin{equation}\label{tcov1/2}
\tc \sim \brac{ \frac 1{2\log 2} \log \frac 1{\e}} n \log n,
\end{equation}
so if    $d\e = \ell \log d$, $\ell$ constant, then $\tc \sim (1/2\log 2)\,  n \,\log n \,\log \log n$.
On the other hand, if $\e$ is constant then $\tc = \Th(n \log n)$, and as
$p \ra 1$ then $\tc$ tends to $ n \log n$.

\paragraph{Notation.}
 $G=(V,E)$ is the graph with vertex set $V$ and edge set $E=E(G)$, where we take $V=[n]$ throughout.   The degree of a vertex $v \in V$ is denoted by $d_v$. For $S \seq V$,    $\deg(S)$ is the degree of  set $S$, where $\deg(S)= \sum_{v \in S} d_v$ and
  $N(S)=\set{w\notin S:\exists v\in S\ s.t.\ \set{v,w}\in E(G)}$ is the disjoint neighbour set of $S$.
We use $\log x$ for the natural logarithm of  $x$, and $\log_2 x$ for the logarithm base 2.
  The degree of a vertex in the $n$-vertex hypercube $Q_n$ is $d=\log_2 n$.

We use $t=0,1,...$ to index time steps,  reserve $T$ for a mixing time, and $\tc$ for cover time.
We assume $d\e$ is integer, and if not use the term \lq vertices of degree $d\e$\rq ~to denote the  vertices of degree $\rdown{d\e}$ and $\rdup{d\e}$. We use $dist(u,v)$ as the minimum distance between vertices $u,v$ of a graph.

A sequence of events $\cE_n$ occurs {\em with high probability},
(w.h.p.), if
$\lim_{n\to\infty}\Pr(\cE_n)=1$. We use the standard notation $O(\cdot), o(\cdot)$ etc, this denoting $o_n(\cdot)$ and so on.
We use $A_n \sim B_n$ to denote $A_n=\ooi B_n$ and thus $\lim_{n \rai} A_n/B_n=1$.
We use $\om$ to denote a quantity which tends to infinity with $n$ more slowly than any other functions  in the given expression.
The expression $ f(n) \ll g(n)$ indicates $f(n)=o(g(n))$.

\section{Background to cover time proof}
\subsection{The first visit time lemma}\label{fv lemma}
Let $G=(V,E)$ be a connected $n$-vertex graph with $m=|E|$ edges. Let $u\in V$ be arbitrary. Let $\cW_u$ denote the random walk $(X(t), {t \ge 0})$ starting from $X(0)=u$. The walk defines a reversible Markov chain with state space $V$. Let $P$ be the matrix of transition probabilities, and $\pi_v=d_v/2m$ the stationary distribution of $P$.
Considering a  walk $\cW_v$, starting
at $v$, let $r_t=\Pr(X(t)=v)$ be the probability  the  walk
returns to $v$ at step $t \ge 0$, and thus $r_0=1$.
Let $R(z)$ generate the sequence $(r_t, t \ge 0)$, and $R(t, z)$ generate the first $t$ entries, $(r_0,\ldots,r_{t-1})$. Thus
\[
R(z)=\sum_{t=0}^\infty r_tz^t, \qquad\qquad R(t, z)=\sum_{j=0}^{t-1} r_jz^j.
\]
Finally, for a fixed value of $T$ to be specified,  let $R_v=R(T,1)$, and note that $R_v \ge r_0=1$.

The following first visit time lemma bounds the probability a vertex has not been visited at steps $T, T+1, \dots, t$.
\begin{lemma} \label{firstvisit}{\sc The first visit time lemma} \cite{CF}\\
Let $G$ be a graph satisfying the following conditions
\begin{enumerate}[(i)]
\item \label{mixing time} For all $t\geq T$, $\max_{u,x \in V} |P_u^{(t)}(x)- \pi_x|\leq n^{-3}$.
\item \label{rt} For some (small) constant $\theta>0$ and some (large) constant $K >0$,
\[ \min_{|z|\leq 1+ \frac{1}{KT}} |R(T,z)| \geq \theta.
\]
\item $T\pi_v=o(1)$ and $T\pi_v = \Omega(n^{-2})$.
\end{enumerate}
Let $\mathcal{A}_v(t)$ be the event that the random walk $\mathcal{W}_u$ on graph $G$ does not visit vertex $v$ at steps $T, T+1, \dots, t$. Then, uniformly in $v$,
$$\Pr(\mathcal{A}_v(t)) = \frac{(1+O(T\pi_v))}{(1+p_v)^t} +O(T^2 \pi_v e^{-t/KT})$$
where $p_v$ is given by the following formula, with $R_v= R_v(T,1)$:
$$
p_v= \frac{\pi_v}{R_v(1+O(T\pi_v))}.
$$
\end{lemma}

For the cover time of $Q_{n,p}$ we use the following w.h.p. values of the parameters  in Lemma \ref{firstvisit}.
The total degree  $2m =(1+O( 1/\sqrt n) ) dnp$, and $\pi_v=d_v/2m$, where $1 \le d_v \le d$.
 The value of
$T=O(\log^k n)$ for some  constant $k \le 7$, and thus $T\pi_v=O(\log^k n/n)$.
The value of
$R_v=1+O(1/\log d)$, so $p_v=(d_v/2ndp)(1+O(1/\log d))$.

As we consider values of  $t \ge n \log n$,  there are  values $\nu_1,\nu_2  =O(1/\log d)$, and  $\nu_1 \le \nu_2$ such that  $1-\nu_1 \le dnp/2mR_v \le 1-\nu_2$, and
\[
e^{-(1-\nu_1)\d_v t/dnp} \le \Pr(\mathcal{A}_v(t)) = (1+O(T\pi_v)) e^{-t\pi_v/R_v} \le  e^{-(1-\nu_2)d_v t/dnp}
.
\]
To tidy things up, write
\beq{defmu}{
\frac{dnp}{2mR_v} =1-\nu\text{ where }\nu=O(1/\log d),
}
 is to be understood as a variable which abbreviates the  inequality  $\nu_1 \le \nu \le \nu_2$, and
\beq{PrAt}{
\Pr(\mathcal{A}_v(t))  = e^{-(1-\nu)d_v t/dnp}.
}
The value of $T$ and Condition \eqref{mixing time} of the first visit time lemma will be established in Section \ref{Tmix}.
The claim that $R_v =1 +O(1/\log d)$ is proved in Section \ref{Rv=1}.
Condition (iii) holds as $1 \le d_v \le d$ and $m \sim dnp$ where $p \ge 1/2$.
 We rely on the following lemma (Lemma 18 of \cite{Hyp}) to establish Condition \eqref{rt}.

\begin{lemma}\label{helper}
 Let $v$ be a vertex of a connected $n$-vertex graph $G$. Let $T$ be a mixing time satisfying Condition \eqref{mixing time}
  of Lemma \ref{firstvisit}. If $T= o(n^3)$, $T\pi_v=o(1)$ and $R_v$ is bounded above by a constant, then Condition \eqref{rt} of Lemma \ref{firstvisit} holds for $\theta= 1/4$ and any constant $K \geq 3R_v$.
\end{lemma}

\subsection{Properties of $Q_{n,p}$ used in the proofs}

Vertices of degree $d\e$ have a particular significance in the proofs, as values around $d\e$ determine the cover time of the random walk.
For convenience we assume $d\e$ is integer, and if not take this to mean the union of vertices of degree $\rdown{d\e}$ and $\rdup{d\e}$.

For $p \ge p_c$ the following  properties of $Q_{n,p}$ hold w.h.p.
\begin{enumerate}[P1.]
\item {\sc Conductance.} \label{P1} The conductance of $Q_{n,p}$ is
$
\F=\Omega\bfrac{1}{d^3\log d}$.

\item {\sc Minimum degree.}  \label{P2}
For
$d\e> \th \log d$, the minimum degree at least one.

\item {\sc Distance between low degree vertices.}
 \label{P3} Let $S_L=\{v \in V: d_v \le L\}$. A vertex $v$ is of low degree if $d_v \le L$, given in \eqref{hL} below.
  Fix the values of $h,L$ to
\begin{equation}\label{hL}
h= \frac{d}{2 \log d}, \qquad L=  \frac {100d}{\log d}.
\end{equation}
No two vertices of degree at most $L$ are within distance $h$ of each other.\\

\item  {\sc Degree of last to be visited vertices.} \label{P4}
 Vertices of degree $\sim d\e$ are last  to be visited.  

\item \label{P5} {\sc The number of vertices degree $d\e$.} The number $X(d\e)$ of vertices of degree $d\e$ satisfies $X(d\e)=\E X(d\e) \ooi$, where $\E X(d\e)$ is given by \eqref{Edex}.

\item {\sc Distance between vertices of degree $d\e$.} \label{P6} If $\e \le 1/100$ no two vertices of degree $d\e$ are within distance $h$ of each other.
\end{enumerate}
The proofs of these properties are given in the Appendix;
P1 in Section \ref{Conduc}, P2 in Lemma \ref{MinDeg}.\ref{Part1},
P3 in  Lemma \ref{LowDeg}.\ref{LD1}, P4 in  Section \ref{Last}, P5 in  Section \ref{Conc-Var}, and
P6 in Lemma \ref{LowDeg}.\ref{LD2}.

\subsection{Mixing time of the random walk}\label{Tmix}
The {\em conductance}
$\F(G)$ of a graph $G$ is
\[
\F=\min_{\p(S)\leq 1/2}\frac{|E(S:\ol{S})|}{\deg(S)}.
\]
Here $\deg(S)=\sum_{v\in S}d(v)$, $\p(S)=\frac{\deg(S)}{\deg(G)}$,
and $E(S:\ol{S})$ is the set of edges between $S$ and $V \sm S$ in the
$G$. It follows
from  \cite{JS} that
\begin{equation}\label{mix}
|P_{u}^{(t)}(x)-\pi_x| \leq (\p_x/\p_u)^{1/2}(1-\Phi^2/2)^t.
\end{equation}
As we assume $Q_{n,p}$ is connected and the maximum degree is $d$ we have $\p_x/\p_u=O(\log n)$.
It follows from P1 (see Section  \ref{Conduc} for the proof)  that  $\F=\Om(1/\log^3 n \log \log n)$.
To  satisfy Condition \eqref{mixing time} of Lemma \ref{firstvisit}, we  take
\begin{equation}\label{mixx}
T=\log^7 n.
\end{equation}

A walk is {\em lazy}, if it only
moves to a neighbour with probability $1/2$ at any step.
There are several technical points in our cover time proof which require us to consider lazy walks.
Firstly the hypercube is bipartite, and hence periodic.
To remove the periodicity we
can make the walk lazy.
Secondly the bound \eqref{mix}
assumes  the walk is lazy.

Making the walk lazy halves the conductance but \eqref{mixx} still holds, and the value of $\pi_v$
is unchanged.
Using a lazy walk asymptotically doubles the cover time, as half the steps are wasted.
It also  doubles the value of $R_v$; as the expected number of steps before an exit from $v$ is two.
Thus the ratio of these values cancels in \eqref{PrAt}.
Other then this it has a negligible effect on the
analysis, and we will ignore it for the rest of the paper and continue
as though there are no lazy steps.

\subsection{The number of returns in the mixing time}\label{Rv=1}

For a random walk $X_t$ starting from a vertex $v$ of a graph $G$, let $R_v(T)$ denote the expected number of visits to $v$ in $T$ steps. As $X_0=v$, we have
\[
R_v(T)=1+ \sum_{t=1}^T \Pr(X_t=v).
\]
\begin{lemma}\label{RT}
Let $p \ge p_c$, and let $X_t$ be a random walk on $Q_{n,p}$. Then w.h.p. for all $v \in V$ and all $T =O(\log^k n)$, $k $ constant, $R_v(T)=1+O(1/\log d)$.
\end{lemma}
\begin{proof}
As in \eqref{hL} of P3, fix the values of $h,L$ to $h= \frac{d}{2 \log d}$ and $ L=  \frac {100d}{\log d}$.
Let $t_0=d/\log^2 d$.
The proof is in three steps, from $t \le  t_0$, from $t_0<t \le h$, and from $h \le t \le T=\log^k n$.

We first consider the case where, with the possible exception of $v$ itself, no vertex within distance $h$ of $v$ has degree at most $L$.
\begin{align*}
\sum_{\t=1}^{t_0} \Pr(X_\t=v) \le& \sum_{\t=1}^{t_0} \frac 1L \Pr(X_{\t-1}\in N(v))
\le \frac{t_0}{L} = O \bfrac 1{\log d}.
\end{align*}
Let $v_i$ be a vertex at distance $1 \le i \le h$ from $v$. The probability the distance to $v$ decreases to $i-1$ at the
next step is at most $i/L$, and the probability it increases to $i+1$ is at least $(L-i)/L$. Thus the drift
away from $v$ per step is at least
\[
\mu \ge \frac{L-h}{L}- \frac hL = \frac{L-2h}{L} = \frac{99}{100}.
\]
In $t \le h$ steps the expected displacement of the walk from $v$ is at least $t\mu$. Let $dist(X_t,v)$ be the
actual displacement. For $\d>0$ constant
\[
\Pr(dist(X_t,v)\le (1-\d)t \mu) \le e^{-\Omega(\d^2 t\m)}.
\]
Thus
\[
\sum_{\t=t_0}^h \Pr(X_\t=v) \le he^{-\Omega(\d^2 t_0\m)} =o\bfrac{1}{\log d}.
\]
It follows  that, w.h.p., in $h$ steps the walk is at least distance $H=\mu h(1-2\d)$ from $v$,
where $H \ge 2 d/(5 \log d)$, say. Let
\[
\wh q= \frac{L-h}{L}= \frac{199}{200}, \qquad \wh p= \frac hL= \frac 1{200}.
\]
Consider a biassed random walk  with transition probabilities $\wh p$ of one step left and $\wh q$ of one step right,
setting out from $H-1$ on the integer line $\{0,1,...,H\}$. The probability the walk reaches the origin $v$ before returning to $H$ is
\[
\frac{ \bfrac{\wh q}{\wh p}-1}{ \bfrac{\wh q}{\wh p}^H-1}=O(\mu^{-h})=O\brac{ \bfrac{199}{200}^{2d/5\log d}}.
\]
Thus, with $T=\log^k n$, for any constant $k$,
\[
\sum_{\t=h}^{T} \Pr(X_\t=v) \le O\brac{ \bfrac{199}{200}^{2d/5\log d}}=o\bfrac{1}{\log d}.
\]

Next consider the case where  vertex $w$ is one of the at most 2 vertices of degree at most $L$ is within distance $h$ of $v$.
If $w \in N(v)$ this can increase the expected returns by $O(1/L)$. Suppose $w$ is a distance $i\ge 2$ from $v$. In the worst case assume the walk always returns to level $i-1$ (a wasted move). Deleting all edges between $w$ and its neighbours leaves all vertices within distance $h$ of $v$ with degree at least $L-1$. This has a negligible effect on the analysis given above.
\end{proof}

\section{The cover time of $Q_{n,p}$. Proof of Theorem \ref{TH1}}

\subsection{Proof outline} \label{Pout}
Before proceeding we give a quick sketch of the upper and lower bound proofs, as this will motivate the subsequent calculations.

Let $X_p(i)$ be the number of vertices of degree $i$ in $Q_{n,p}$. Let $q=1-p$, then
\begin{equation}\label{EXp(i)}
\E X_p(i)=n {d \choose i}p^iq^{d-i}.
\end{equation}
Recall that $\mathcal{A}_v(t)$ given in \eqref{PrAt} is an upper bound on the probability vertex $v$ is unvisited at step $t$. Let $S(t)$ be the vertices \lq still surviving\rq ~at step $t$.
\[
S(t)= \sum_{v \in V} \Pr(\mathcal{A}_v(t))\sim \sum_{v \in V} e^{-d_v t/dnp}.
\]
Thus
\begin{align*}
\E S(t) \sim& \sum_{i \ge 1} \E X_p(i) e^{-it/dnp}\;\;\sim \; n(1-p+pe^{-t/ndp})^d.
\end{align*}
Put $t=\a ndp$ and equate $\E S(t)=1$. Using $d=\log_2 n=\log_e n/\log_e 2$,
\[
\log  \E S \sim \log n+d\log (1-p+pe^{-\a})= \frac{\log n}{\log 2}( \log 2 +\log (1-p+pe^{-\a})).
\]
\begin{align*}
\log \E S=0\iff& \log 2 +\log (1-p+pe^{-\a})=0
\iff (1-p+pe^{-\a})=1/2.
\end{align*}
Solving this gives $ \a = \log {2p}/({2p-1})$, which  suggests the following result.
\begin{equation}
\tc\sim ndp \cdot \log \frac{2p}{2p-1}. \nn
\end{equation}
For the corresponding lower bound we prove that vertices of degree $d\e$   maximize the above calculations,
where $d\e$ is somewhat larger than the minimum degree.
We  prove by direct construction that w.h.p. at some step $t$ slightly below  $\tc$
there are many vertices of degree $d\e$ which are unvisited by the walk. We now proceed to the details of the above proof idea.
\subsection{Upper bound on the cover time}

Let $T(u)$ be the time taken by the random walk $\cW_u$
to visit every vertex of a connected graph
$G$, and $\tc(u)=\E T(u)$. Let $U_t$ be
the number of vertices of $G$ which have not been visited by
$\cW_u$ at step $t$.
We note the following:
\begin{eqnarray}
\gap{.4}\tc(u)=\E(T(u))&=& \sum_{t > 0} \Pr(T(u) \ge t), \label{ETG} \\
\label{TG}
\Pr(T(u)\geq t)=\Pr(T(u) > t-1)&=&\Pr(U_{t-1}>0)\le \min\{1,\E(U_{t-1})\}.
\end{eqnarray}
As in  \eqref{PrAt}, let $\mcA_v(t),\,t\geq T$ be the event that $\cW_u(t)$ has
not visited $v$ in the interval $[T,t]$.
It follows from (\ref{ETG}), (\ref{TG}) that for all $t\geq T$,
\begin{equation}\label{tcu}
\tc(u) \le t+1+ \sum_{s \ge t} \E(U_{s})\le t+1 +\sum_{v\in V} \sum_{s \ge t}\Pr(\mcA_s(v))
\end{equation}
and
\begin{equation}\label{shed1}
\sum_{s \ge t}\Pr(\mathcal{A}_v(s)) \le \sum_{s \ge t} e^{-(1-\nu)d_v s/dnp} \le \frac{dnp}{(1-\nu)d_v}e^{-(1-\nu)d_v t/dnp},
\end{equation}
where $\nu=O(1/\log d)$.
Let $X(i)$ be the number of vertices of degree $i$. The above argument implies that
\begin{equation}\label{sget}
Q(t)=\sum_{v\in V}\sum_{s \ge t} \Pr(\mathcal{A}_s(v)) =
O\brac{dnp}\sum_{i=1}^d{X(i)} e^{- (1-\nu)it/dnp}.
\end{equation}

The argument given in Section \ref{Pout} can  now be adapted to give an upper bound on the cover time.
Let $X(i)$ be the number of vertices of degree $i$ in $Q_{n,p}$.
Let $b=d \om$, then
\begin{equation}\label{bew}
\Pr(X(i) \ge b \E X(i)) \le \frac 1b,
\end{equation}
so with probability $1-O(1/\om)$ this  upper bound of $d \om \E X(i)$ on the number of vertices of
degree $i$ holds simultaneously  for all $i \in \{0,1,...,d\}$.

Let $\d= (\log db)/\log n=o(1)$, and $t_U= \a \;(ndp)/(1-\nu)$ where
\begin{equation}\label{tu}
\a=
 \log \frac{p}{p-1+ \brac{\half}^{1+\d}} = \log \frac{2p}{2p-1 -O(\d)}.
\end{equation}

Let $S(t)$ be the number of unvisited vertices at $t$. By the above estimate  \eqref{EXp(i)}, \eqref{sget}, \eqref{bew}, and with $t_U=\a dnp/(1-\nu))$, then w.h.p.,
\begin{align}\label{SQ}
S(t_U)=&\;\sum_{v \in V} \Pr(\mathcal{A}_v(t_U))\\
\le &\;b \sum_{i \ge 1} \E X(i) e^{-it_U(1-\nu)/dnp} =  bn (1-p+pe^{-\a})^d,\\
Q(t_U)=&\; O(dnp)\; S(t_U)\;= \;O(dnp) bn (1-p+pe^{-\a})^d.
\end{align}
However, by \eqref{tu}
\[
(1-p+pe^{-\a})^d=  2^{-(1+\d)d}=n^{-(1+\d)},
\]
so
\[
Q(t_U)=O(1) n^2 db n^{-(1+\d)}=O(n).
\]
Thus by \eqref{tcu}, for any $u \in V$,
\[
\tc(u)\le t_U+1+Q(t_u)=t_U+O(n)
\]
Finally
\[
\tc \le (1+o(1))\; ndp \;\log \frac{2p}{2p-1}.
\]

\subsection{Lower bound on the cover time}\label{seclob}

Let $S(0)$ be the set of vertices of degree $d\e$ in $Q_{n,p}$. We construct a subset of $S(0)$ which is still unvisited w.h.p. at $t_L= t_U(1-o(1))$.
By \eqref{Edex} below
 \beq{ES0}{
\E |S(0)|=\E X(d\e) \sim
\frac{1}{\sqrt{2\pi d\e(1-\e)}} \; \bfrac{1+\e}{\e}^{d \e },
}
and by property P\ref{P5}, $|S(0)| =\ooi \E S(0)$.

The function $f(x)=((1+x)/x)^x$ is monotone increasing from one for $x \in (0,1]$.   Thus for any $d\e \ge \th \log d$, (i.e., $p \ge p_c$, see paragraph preceding Theorem \ref{TH1}), the value of $|S(0)|$ is much greater than $T=\log^7 n$, as given in \eqref{mixx}. We remove any vertices visited during $T$ from $S(0)$ to apply the results of Lemma \ref{firstvisit}.

Let $t_L=(1-\d)ndp \log 2p/(2p-1)$ where $\d=o(1)$ is given by \eqref{delta-val} below. At step $t_L$, given the value of $|S(0)|$,
\begin{equation}\label{Stde}
\E |S(t_L)| \sim  |S(0)| e^{-(1-\nu) d \e t_L/dnp}\sim \frac{1}{\sqrt{2\pi d\e(1-\e)}} \; \bfrac{1+\e}{\e}^{\d d\e (1-\nu)},
\end{equation}
where the expectation is with respect to the random walk.

\paragraph{Case where  $\e \le 1/100$.}
It follows from Lemma \ref{LowDeg} that w.h.p. for $\e \le 1/100$ all vertices of degree $d\e$ are at least a distance $h=d/2\log d$ apart. Choose two vertices $v,w \in S$, let the graph distance between them be $\ell \ge h$. Coalesce these into a single vertex $\g=\g(v,w)$, to form a graph $\G(v,w)$. We claim $R_\g=1+O(1/\log d)$. The proof is similar to that of Lemma \ref{RT}. Project the walk starting from $\g$ onto an integer line length $h/2$, with $\g$ identified with zero, and a loop at $h/2$.

Let $Y_x=Y_x(t)$ be the indicator that $\cW_u$ has not visited vertex $x$ at $t$. As $R_v, R_w, R_\g=1+O(1/\log d)$, and $d_\g=2d\e$ it follows that
\[
\E Y_vY_w=e^{-(1-o(1)) 2d\e t/ndp}=\brac{ e^{-(1-o(1)) d\e t/ndp}}^2 =\ooi \E Y_v \E Y_w,
\]
and so
\begin{equation}\label{ESt^2}
\E |S(t)^2|=|S(0)|(|S(0)|-1)e^{-(1+o(1))2d\e t/ndp}+|S(0)|e^{-(1+o(1))d\e t/ndp}.
\end{equation}
Choose $\d$ so that $\sqrt d \e^{\d d\e} =o(1)$.
This is satisfied by
\begin{equation}\label{delta-val}
\d= \frac{\log d}{d\e \log 1/\e}= O\bfrac{1}{\log \log d}=o(1).
\end{equation}
By \eqref{Stde} and \eqref{ESt^2},
\[
\Pr(|S(t_L)| \ne 0) \ge (1-o(1)) \frac{(\E |S(t_L)|)^2}{\E |S(t_L)|^2}= 1-\frac {O(1)}{\E |S(t_L)|}.
\]
Using \eqref{delta-val} in  \eqref{Stde}, we see that $\E |S(t_L)| \rai$ and thus $\Pr(|S(t_L)| \ne 0)=1-o(1)$ as required.

\paragraph{Case where  $\e \ge 1/100$.}
Let $F_\ell(v)=\{w: \dist(v,w)\le\ell\}$; where $\dist(v,w)$ is graph distance in $Q_n$, and $\ell$ is to be determined.   Let $V^*$ be some maximal set of vertices of $Q_n$ such that for all $u,v \in V^*$, $\dist(u,v)> \ell$. Then $|F_\ell(v)|\le d^\ell$, and so $|V^* |\ge n/2d^\ell$.

Let $B(0)=\{v: v \in V^*,\, d_v=d\e\}$, where $d_v$ is the degree of $v$ in $Q_{n,p}$. Then $B(0) \seq S(0)$ as defined above, and as w.h.p. $|B(0)| =\ooi \E |B(0)|$,
\[
|B(0)| \ge |V^*| \frac 1n \frac 1{3 \sqrt d} \bfrac{1+\e}{\e}^{d\e} \ge \frac 1{6d^{\ell+1/2} } \bfrac{1+\e}{\e}^{d\e}.
\]
Thus at $t=t_L$,
\[
\E |B(t)| \ge \frac{1}{6d^{\ell+1/2}} \bfrac{1+\e}{\e}^{\d d\e},
\]
where we require $\E |B(t)|  =\om\to\infty $, say. This is satisfied for large $\ell$ by any
\[
\d \ge \frac{ 10}{\e \log(1+1/\e)} \frac{\ell \log d}{d} \ge  \frac{C\ell \log d}{d},
\]
for some constant $C$, as   $\e \ge 1/100$.
Choose $\ell =d/(\om \log d)$,
then $\d=O(1/\om)$ and $\E |B(t_L)| \rai$ as required.
With this value of $\ell$, for any pair $v,w \in B(t)$
an argument similar to Section \ref{Rv=1},  (with the simplifying fact from Lemma \ref{MinDeg}.\ref{part4}, that if $\e$ is constant, then $\d=\a_0 d\e$ for some constant $\a_0 \in (0,1)$), that will ensure that $R_{\g(v,w)}=1+o(1)$.
The rest of the proof is  similar to the previous case.

\section{Appendix: Conductance and other technical details}
\subsection{Conductance of $Q_{n,p}$} \label{Conduc}

The conductance $\F=\F_G$ of a graph $G=(V,E)$ is defined as
\[
\F_G= \min_{S \subset V(G) \atop 0< \pi(S) \le 1/2} \frac{e(S:\overline S)}{d_G(S)},
\]
where $d(S)=\sum_{v\in S}d_v$ is the degree of a set of vertices $S$ in the graph $G$, $\ol S=V \sm S$ and $ e(S:\overline S)=|E(S:\overline S)|$.  The expression $\pi(S) \le 1/2$ is equivalent to $d(S)\leq |E(G)|$; multiply the former by $2|E(G)|$ to obtain the latter.

The edge isoperimetric inequality for the hypercube, Harper \cite{Har}, states that
\begin{equation}\label{Cond}
\min _{S\subseteq V
\atop |S| \le n/2}\left\{|E(S,{\overline {S}})|\right\}
\geq |S|(d-\log _{2}|S|).
\end{equation}
The bound is tight for sets
$S$ which are vertices of a subcube of $Q_n$.
For random sub-hypercubes we have the following lower bound.
\begin{proposition}\label{dubious}  With high probability $Q_{n,p}$ has conductance
\[
\F_{Q_{n,p}}=\Omega\bfrac{1}{d^3\log d}.
\]
\end{proposition}

\begin{proof}
{\bf Case 0.} $1 \le |S| \le \sqrt d$. By Lemma \ref{LowDeg}.\ref{LD1},   w.h.p. vertices of degree at most $L=100 d/\log d$ are distance at least $d/2 \log d$ apart. Let $S_1 \seq S$ be vertices of degree at most $L$ and $S_2=S\sm S_1$.
Then
\[
e(S_1, \ol S)=d(S_1) \quad \text{ and }\quad e(S_2, \ol S) \ge 100|S_2| d/\log d- 2|S_2| \log_2 |S_2| \ge |S_2| d/\log d,
\]
and $d(S) \le d(S_1)+ d|S_2|$. It follows that
\begin{equation} \label{smalldeg}
\F_S \ge \frac{1}{\log d}.
\end{equation}

{\bf Case 1:} $\sqrt d\leq |S|\leq n/3d$.
Referring to \eqref{Cond}, as $p>1/2$ the number of retained edges $e(S:\ol S)$ is at least
 $X \sim Bin(s(d-\log_2s),1/2)$. Thus,
\begin{align}
\Pr(\exists S: \sqrt d \le |S|\leq n/3d, e(S:\ol S)&\leq s(d-\log_2s)/d) \nn\\
&\leq\sum_{s=\sqrt d}^{n/3d}   n(ed)^{s-1} \;\Pr(X\leq s(d-\log_2s)/d)\label{QTree}\\
&\leq \sum_{s=\sqrt d}^{n/3d}n(ed)^s\frac{(ed)^{s(d-\log_2s)/d}}{2^{s(d-\log_2s)}}\label{bin}\\
&= \sum_{s=\sqrt d}^{n/3d} \brac{\frac{s \,2^{d/s} }{(ed)^{(\log_2s)/d}} \frac{(ed)^2}{n}}^{s} \label{eq1}\\
 &\le \sum_{s=\sqrt d}^{n/3d}  \bfrac{e^{1+o(1)}}{3}^s= o(1) . \nn
\end{align}
In \eqref{QTree} we used  the estimate $(ed)^{s-1}$ as an upper bound on the number of trees  of size $s$ in $Q_n$, rooted at a fixed vertex, see \cite{BFMcD}. If $S$ induces more than one component this can only increase the number of edges to $\overline S$.
Equation \eqref{bin} used the following.
\begin{equation}\label{BinSum}
\bfrac{k}{N}^k\sum_{i=0}^k\binom{N}{i}\leq \bfrac{k}{N}^k\sum_{i=0}^k \frac{N^i}{i!}=\sum_{i=0}^k\frac{k^i}{i!}\bfrac{k}{N}^{k-i}\leq e^k.
\end{equation}
Here $N=s(d-\log_2s)$   and  $k=N/d$. The  bracketed term in \eqref{eq1} has a unique minimum at $s=d \log 2/(1-(\log_2ed)/d)>n/3d$. Therefore the maximum value in the bracket in the sum occurs at $s=n/3d$.

If $s\le n/3d$, then  $s(d-\log_2s)/d\geq s\log_2d/d$. Thus for $|S| \le n/3d$,
\begin{equation}\label{n/3d}
\frac{e(S:\ol S)}{d(S)}= \Om \bfrac{\log_2 d}{d^2}.
\end{equation}

{\bf Case 2:} $|S | \ge n/3d$.
It follows from \cite{AKS}, and  Theorem 1.4  of \cite{EKK} respectively that given $\d>0$
there exists constants $c_1,c_2>0$ such that if $q=\frac{c}{d},c\geq c_1$ then $Q_{n,q}$ contains a subgraph $H$ such (i) $|V(H)|\geq (1-\d)n$ and (ii) $H$ is a $(c_2/d^2 \log d)$-expander.  A graph $G$ is an $\a$-expander if $|N(S)|\geq \a|S|$ for all $S\subseteq V(G)$ for which $|S|\leq |V(G)|/2$.

By the above, $Q_{n,p}$ contains the union of $ h\sim dp/c$ independent and uniformly chosen vertex subsets $H_1,H_2,\ldots,H_h\subseteq Q_n$, each of which induces an expander.  Let $\G=\cup_{i=1}^h H_i$,
so that $\G \seq Q_{n,p}$.
The graph $\G$ and each independent copy $H=H_i$ have the following properties w.h.p.:
\begin{enumerate}[{P(i).}]
\item $|E(H_i)|\sim \tfrac12 cn$.
\item $\displaystyle {\sum_{v:d_H(v)\notin [.99c,1.01c]}}d_H(v)\leq ne^{-\Omega(c)}$.
\item  $\displaystyle {\sum_{v:d_\G(v)\notin [.99dp,1.01dp]}}d_\G(v)\leq ne^{-\Omega(d)}$.
\item $|E(\G)|\sim \tfrac12 ndp$.
\end{enumerate}
We need to estimate
\[
\F_\G=\min\set{\F_S:d_\G(S)\leq |E(\G)|},\text{ where }\F_S=\frac{e{_\G}(S:\ol S)}{d_\G(S)}.
\]
It follows from P(iii) that for $n/3d\leq |S|\leq 2n/3$
\[
0.98dp|S|\leq 0.99dp|S|-ne^{-\Omega(d)}\leq  d_\G(S) \leq 1.01dp|S|+ne^{-\Omega(d)}\leq  1.02dp|S|.
\]
So $d_\G(S)= \k d|S|$ for some constant $\k$,
$0.98 p \le \k \le 1.02 p$,  and  $d_\G(S)\leq |E(\G)|$ holds for $|S|\leq 3n/5$.

A similar argument using P(ii) implies that if $S\subseteq V(H_i)$ and $n_i=|V(H_i)|$ then
\[
0.98c|S|\leq 0.99c|S|-n_ie^{-\Omega(c)}\leq  d_\G(S) \leq 1.01c|S|+n_ie^{-\Omega(c)}\leq 1.02c|S|.
\]
Also, with $h=dp/c$  and $\d=.0001$,
\begin{align}
\Pr(n/3d\leq |S|&\leq 3n/5: |S\cap V(H_i)|<0.99|S|\text{ for all values } i=1,...,h)\nn\\
&\leq\sum_{s=n/3d}^{3n/5}\binom{n}{s}[\Pr(Bin(s,1-\d)<0.99s)]^{h}\nn\\
&\leq\sum_{s=n/3d}^{3n/5} \bfrac{ne}{s}^se^{-\Omega(ds/c)}=\sum_{s=n/3d}^{3n/5}\brac{\frac{ne}{s}e^{-\Omega(d/c)}}^s=o(1).\label{eq2}
\end{align}
It follows that if $n/3d\leq |S|\leq 3n/5$, there exists $H_i$ such if $T=S\cap V(H_i)$ then $|T|\geq 0.99|S|$.

Next let $T'$ be the smaller of $|T|,|V(H_i)\setminus T|$ and note that $|T'|\geq |S|/2$.
By Theorem 1.4 of \cite{EKK} the set $T'$ has at least ${c_2|T'|}/{d^2\log d}$ neighbours in $|V(H_i)\setminus T|$, and
as by definition $S \sm T$ is disjoint from $H_i$,
\[
e(S:\ol S)\geq e(T':V(H_i)\setminus T')\geq \frac{c_2|T'|}{d^2\log d}\geq \frac{c_2|S|}{2d^2\log d}.
\]

In summary, for $n/3d \le $S$ \le 3n/5$,
\[
\frac{e(S:\ol S)}{d_\G(S)}=\Omega\bfrac{1}{d^3\log d}.
\]
The  claim of Proposition \ref{dubious} then follows from \eqref{smalldeg}, \eqref{n/3d} and the above.
\end{proof}

\subsection{Various supporting lemmas} \label{VariousBits}

\paragraph{Minimum degree: General bounds}
\begin{lemma}\label{MinDeg}
The following hold w.h.p. in $Q_{n,p}$, for $p=(1+\e)/2$.
\begin{enumerate}
\item \label{Part1} If $d\e=\om \rai$  there are no vertices of degree zero.
Moreover, if $d\e=(i-1 +\th) \log d$, where  $i$ is a fixed integer and $\th \in (0,1)$ constant, the minimum degree $\d$ is $i$.
\item \label{part4}If $\e$ is constant then $\d\ge \a_0 d\e$ for some constant $\a_0 \in (0,1)$.
\end{enumerate}
\end{lemma}
\begin{proof}
{\bf Case 1.} The expected number of vertices of degree zero is $n q^d$ which tends to zero for any $d\e=\om \rai$.
 Let $X_j$ denote the number of vertices of degree $j$. Then
\beq{FXj}{
\E X_j= \;n {d \choose j}p^jq^{d-j} = \binom{d}{j}(1+\e)^j(1-\e)^{d-j}
}
So, if $j\leq i-1$ then
\[
\E X_j\leq \bfrac{de(1+\e)}{(1-\e)j}^jd^{-(i-1+\th)}=o(1).
\]
Whereas
\[
\E X_i\geq \bfrac{d(1+\e)}{(1-\e)i}^id^{-(i-1+\th)}\to\infty.
\]
An application of the Chebychev inequality will show that $X_i>0$ w.h.p.

\ignore{
{\red {\bf Case 2.} Suppose that $d\e=\a\log^kd$. then if $j=\b\log^{k-1}d$ then
\[
\E X_j\leq \bfrac{de(1+\e)}{(1-\e)\b\log^{k-1}d} ^{\b\log^{k-1}d}e^{-\a\log^kd}=o(1),
\]
if $\a>\b$.

If $\a<\b$ then $\E X_j\to\infty$ and we can use Chebyshev.
}
}

{\bf Case 2.}
Putting $j=\a d\e$, where $\a<1/3$, we obtain from \eqref{FXj} that
\[
\E X_j\le \binom{d}{\a d\e}e^{-(d-2\a)\e^2}\leq \bfrac{e}{\a \e}^{\a d\e}e^{-d\e^2/3}\leq e^{-d\e^2/4},
\]
for small $\a$. Taking the union bound over at most $d$ values for $j$, we see that the minimum degree is at least $\a d\e $ w.h.p. for some small $\a>0$ constant.
\ignore{
On the other hand,
\[
\E X_j\geq \bfrac{1}{\a \e}^{\a d\e}(1+\e)^{\a d\e}(1-\e)^{(1-\a\e)d}\to\infty,
\]
if $\a\e$ is close to 1. Using Chebyshev we see that there are vertices of degree $\a\e d$, implying that the minimum degree satisfies 4.
}
\end{proof}

\paragraph{Low degree vertices.}
The following argument concerns the distance $h$ between low degree vertices.  Fix the values of $h,L$ to
\begin{equation}\nonumber
	h= \frac{d}{2 \log d}, \qquad L=  \frac{100 d}{\log d}.
\end{equation}
Say a vertex $v$ is of \lq low degree' if $d_v \le L$  and let $S_L=\{v \in V: d_v \le L\}$.
For large values of $p$,  by Lemma \ref{MinDeg}.\ref{part4} above,  Lemma \ref{LowDeg}.\ref{LD1} holds with $S_L=\es$. \begin{lemma}\label{LowDeg}
Let  $p=\half(1+\e)$, $ p\ge p_c$, then the following hold w.h.p.:
\begin{enumerate}
\item \label{LD1}
No two vertices of degree at most $L$ are within distance $h$ of each other.
\item \label{LD2} If $\e \le 1/100$,
no two vertices of degree at most $(101/100) d \e$ are within distance $h$ of each other.
\end{enumerate}
\end{lemma}
\begin{proof}
The probability there exist two vertices of $S_L$ are within distance  $\ell \le h$ is
\begin{align*}
P(h) \le & \; n\sum_{i=1}^hd^i \brac{\sum_{\ell\leq L}\binom{d}{\ell}p^\ell q^{d-\ell}}^2\\
&=O(1) nd^h \brac{\binom{d}{L}p^Lq^{d-L}}^2\\
\le& \; \frac{O(1)}{n} \bfrac{dep}{Lq}^{2L} d^h(1-\e)^{2d}\\
\le&\;  \frac{e^{d/2}}{n} \bfrac{e \log d}{50}^{200d/\log d}\;\;
=o(1).
\end{align*}
Indeed, let $u,w \in S_L$  and let $uv_1\cdots v_\ell w$ be a path between them of length $\ell \le h$. The number of  paths length $\ell \le h$ is at most $ hd^h$. The $n$ on the first line upper bounds the number of choices for  $u$, and the last term upper bounds the probability that the vertices $u,w$ have degree at most $L$.
The third line follows from $p/q\le 2$ provided $\e \le 1/3 $, and $d^h=e^{d/2}=n^{1/\log 4}$.

The second case is similar but  requires the further information (see \eqref{Value 100} of Section \ref{Last} below) that the probability a vertex has degree  at most $101 d \e/100$ is at most $\Th(1) ((1+\e)/\e)^{101\e d/100}/n$ in which case the probability $P(h)$ of the stated event  satisfies
\[
P(h)\le  \frac{O(1)}{n} \bfrac{1+\e}{\e}^{202 \e d/100} d^h.
\]
For   $P(h)=o(1)$, we require that
\[
\brac{\log 2-\frac 12 -\frac{202}{100}\e \log \frac{1+\e}{\e}-o(1)} > 0.
\]
The function $x \log (1+x)/x$ is  monotone increasing for $x \in (0,1/(e-1)]$,
so the above condition holds for $\e \le 1/100$.
\end{proof}

 \subsection{Degree of the last to be visited vertices.}\label{Last}
Let
\[
N(i)= \frac{d^d}{i^i(d-i)^{d-i}} \bfrac{p}{q}^i q^d,
\]
so that
\[
\E X_p(i) =N(i)\; n \; \sqrt \frac{d}{2\pi i(d-i)} (1+o_d(1)).
\]
For $p \ge p_c$,  $\e d \ge \th\log d$ (for some $\th>0$ constant). Thus
for $\e<1$, $\sqrt{d/(d\e(d-d\e))}=o(1)$, whereas if $i \ge 1$ constant then $\sqrt{d/i(d-i)}=\Th(1)$.

If $i=dp$ then $d-i=dq$ so
\[
N(dp)=\frac{d^d}{(dp)^{dp}(dq)^{dq}} \; p^{dp} q^{dq} =1.
\]

For $0 \le x <p$, as $d-d(p+x)=d(q+x)$,
\begin{align}\label{*1}
N(d(p-x))=&\; \frac{d^d p^{d(p-x)}q^{d(q+x)}}{(d(p-x))^{d(p-x)}(d(q+x))^{d(q+x)}} \nonumber\\
=&\frac{ p^{d(p-x)}}{(p-x)^{d(p-x)}}
\frac{q^{d(q+x)}}{(q+x)^{d(q+x)}}.
\end{align}
 Now,
\[
N(d\e)= \frac 1{2^d} \bfrac{1+\e}{\e}^{\e d}=\frac 1n \bfrac{2p}{2p-1}^{d(2p-1)}.
\]
Thus
\begin{equation}\label{Edex}
\E X(d\e) \sim \frac{1}{\sqrt{2\pi d\e(1-\e)}} \; \bfrac{1+\e}{\e}^{\e d}.
\end{equation}

Next, for $\a \in (-1,1)$,
\begin{align}
[N(d\e(1-\a))]^{1/d}=& \frac 12 \bfrac{ (1+\e)}{\e(1-\a)}^{\e(1-\a)} \bfrac{(1-\e)}{1-\e(1-\a)}^{1-\e(1-\a)} \nn\\
=& \frac 12 \bfrac{1+\e}{\e}^{\e(1-\a)} \frac 1{(1-\a)^{\e(1-\a)}} \bfrac{1-\e}{1-\e+\a\e}^{1-\e+\a\e}\nn\\
=& \frac 12 \bfrac{1+\e}{\e}^{\e(1-\a)} G_\e(\a). \label{Value 100}
\end{align}
We next prove  that the function $G_\e(\a)$ has a maximum $G_\e(0)=1$ at $\a=0$, and  that
\[
G_\e(\a)=  e^{-\Th(\a^2\e)},
\]
is monotone decreasing from this for $\a \in (-1,1)$.

Let $F(\a)=\log G_\e(\a)$, then $F(0)=0$,
\[
F'(\a)= \e \log \frac{(1-\e)(1-\a)}{1-\e+\a\e}, \qquad F''(\a)=-\frac{\e}{(1-\a)(1-\e(1-\a))}
\]
so $F'(0)=0$ and, provided $\a<1$, for some $\th \in [0,1]$,
\[
F(\a)=(\a^2/2)\cdot F''(\th \a)= -\a^2 \e \Th(1).
\]
Let
$C_\a=1/\brac{2\p d\e(1-\a)(1-\e(1-\a))}^{ 1/2} $.
When $t=ndp \log 2p/(2p-1)$ the value of $e^{-d\e(1-\a) t/ndp}$ is $(\e/(1+\e))^{d\e(1-\a)}$.
Thus the expected  number of  vertices of degree $d\e(1-\a)$ still unvisited at  time $t$ (see \eqref{shed1}) is asymptotic to
\[
n \;C_\a\,N(d\e(1-\a)) e^{-d\e(1-\a) t/ndp}=
 C_\a \;(G_\e(\a))^d= C_\a \; e^{-\a^2 \e d \Th(1)}.
\]
Arguing as in Section \ref{seclob} we see that  vertices of degree  $\sim d\e$ should be last to be visited.

\subsection{Concentration of the number of vertices of a given degree}\label{Conc-Var}
Recall that $X(i)=X_p(i)$ is the number of vertices of degree $i$ in $Q_{n,p}$. The expected value of $X(i)$ is given in \eqref{EXp(i)}.

\paragraph{Variance of $X(i)$.}
Let $X=X(i)$ then
\begin{align*}
\E X(X-1)=& \; n(n-d) \brac{{d \choose i}p^iq^{d-i}}^2 \\
+&\;\;n {d \choose i}p^i q^{d-i}
\left\{ i {d-1 \choose i-1} p^{i-1} q^{d-i} + (d-i) {d-1 \choose i}p^i q^{d-1-i}\right\}.
\end{align*}
The first term is for vertices $v$ at distance at least two  from vertex $u$ in $Q_n$.
The second term is for those vertices at distance one from vertex $u$ in $Q_n$ which are a neighbour of $u$ in $Q_{n,p}$, and those $v$ which are not, respectively.

Thus with $\eta={d \choose i}p^i q^{d-i}$,
\begin{align*}
\E X(X-1)=& \;n^2 \eta^2 -nd \eta^2 +n \eta^2 \brac{\frac{i^2}{dp} + \frac{(d-i)^2}{dq}}\\
=&\; n^2\eta^2 +n \eta^2 \frac{(pd-i)^2}{dpq},
\end{align*}
and
\begin{align*}
\Var X=& \E X(X-1)+\E X-(\E X)^2\\
=& n\eta+ n\eta^2 \frac{(pd-i)^2}{dpq}\\
=& \E X\brac{1+ \E X \frac 1n \frac{(dp-i)^2}{dpq}}.
\end{align*}

\paragraph{Concentration of vertices of degree $d\e$.}
The expected value of $X(d\e)$ is given by \eqref{ES0}.
From the above, as $dp-d\e=dq$
\[
 \Var X(d\e)= \E X(d\e)\brac{1+ \E X(d\e) \frac 1n \frac{dq}{p}}=\ooi \E X(d\e).
\]
The probability that $X(d\e)$ deviates significantly above $\E X(d\e)$ is  therefore
\[
\Pr(X(d\e) \ge \E X(d\e)+ \sqrt {\om \E X(d\e )} )\le  \frac{1+o(1)}{\om}.
\]

\end{document}